\documentclass[a4paper,french,twoside,draft,french,11pt,centertags]{smfartVS}

\usepackage{smfenum,smfthm,amsmath,amssymb,mathrsfs,euscript,color}
\usepackage{stmaryrd,mathabx,upgreek}
\usepackage[latin1]{inputenc}
\input{xypic}

\numberwithin{equation}{section} 
\theoremstyle{plain}

\newcounter{nonumber}

\newtheorem{hypo}{Hypothèse}

\def\A{{\rm A}}
\def\B{{\rm B}}

\def\D{{\rm D}}
\def\E{{\rm E}}
\def\F{{\rm F}}
\def\G{{\rm G}}

\def\I{{\rm I}}
\def\J{{\rm J}}
\def\K{{\rm K}}
\def\L{{\rm L}}
\def\M{{\rm M}}
\def\N{{\rm N}}
\def\O{{\rm O}}

\def\Q{{\rm Q}}
\def\R{{\rm R}}
\def\SS{{\rm S}}
\def\T{{\rm T}}
\def\U{{\rm U}}
\def\V{{\rm V}}
\def\W{{\rm W}}
\def\X{{\rm X}}
\def\Y{{\rm Y}}
\def\Z{{\rm Z}}

\def\Hh{\mathscr{H}}

\def\AA{\mathfrak{A}}

\def\HH{\mathfrak{H}}

\def\a{\alpha} 
\def\b{\beta}

\def\h{\phi}

\def\l{\lambda}
\def\m{\mathfrak{m}}

\def\ie{c'est-à-dire }

\def\>{\geqslant}
\def\<{\leqslant}

\def\Hom{{\rm Hom}}
\def\End{{\rm End}}

\def\GL{{\rm GL}}

\def\Ker{{\rm Ker}}

\def\Ind{{\rm Ind}}
\def\ind{{\rm ind}}

\def\mult#1{{#1}^{\times}}

\def\HT{\EuScript{C}} 
\def\vr{\varrho}
\def\tau{\vr}

\def\Hh{\EuScript{H}}
\def\HH{\EuScript{A}}
\def\rr{\mathfrak{r}}

\author{Guy Henniart}
\address{Université de Paris-Sud\\
Laboratoire de Mathémati\-ques d'Orsay\\
F-91405 Orsay Cedex} 
\email{guy.henniart@math.u-psud.fr}

\author{Vincent Sécherre}
\address{Université de Versailles St-Quentin-en-Yvelines\\
Laboratoire de Mathémati\-ques de Versailles\\
45 avenue des Etats-Unis\\
78035 Versailles cedex, France}
\email{vincent.secherre@math.uvsq.fr}

\title{Types et contragrédientes}

\begin{abstract}
Soit $\G$ un groupe réductif $p$-adique, et soit $\R$ un corps algébriquement 
clos. 
Soit $\pi$ une représentation lisse de $\G$ dans un espace vectoriel $\V$ sur 
$\R$. 
Fixons un sous-groupe ouvert et compact $\K$ de $\G$ et une représentation 
lisse irréductible $\tau$ de $\K$ dans un espace vectoriel $\W$ de dimension 
finie sur $\R$.
Sur l'espace $\Hom_\K(\W,\V)$ agit l'algèbre d'entrelacement $\Hh(\G,\K,\W)$. 
Nous examinons la compatibilité de ces constructions avec le passage aux 
représentations contragrédientes $\V^\vee$ et $\W^\vee$, et donnons en 
particulier des conditions sur $\W$ ou sur la caractéristique de $\R$ pour que 
le comportement soit semblable au cas des représentations complexes. 
Nous prenons un point de vue abstrait, n'utilisant que des propriétés 
générales de $\G$.
Nous terminons par une application à la théorie des types pour le groupe 
$\GL_n$ et ses formes intérieures sur un corps local non archimédien. 
\end{abstract}

\begin{altabstract}
Let $\G$ be a $p$-adic reductive group, and $\R$ an algebraically closed 
field. 
Let $\pi$ be a smooth representation of $\G$ on an $\R$-vector space $\V$. 
Fix an open compact subgroup $\K$ of $\G$ and a smooth irreducible 
representation of $\K$ on a finite-dimensional $\R$-vector space $\W$.  
The space $\Hom_\K(\W,\V)$ is a right module over the intertwining algebra 
$\Hh(\G,\K,\W)$. 
We examine how those constructions behave when we pass to the contragredient 
representations $\V^\vee$ and $\W^\vee$, and we give conditions under which 
the behaviour is the same as in the case of complex representations.  
We take an abstract viewpoint and use only general properties of $\G$.
In the last section, 
we apply this to the theory of types for the group $\GL_n$ and its 
inner forms over a non-Archimedean local field. 
\end{altabstract}

\long\def\MSC#1\EndMSC{\def\arg{#1}\ifx\arg\empty\relax\else
     {\par\narrower\noindent%
     2010 Mathematics Subject Classification: #1\par}\fi}

\long\def\KEY#1\EndKEY{\def\arg{#1}\ifx\arg\empty\relax\else
	{\par\narrower\noindent Keywords and Phrases: #1\par}\fi}

\begin{document}

\thispagestyle{empty}

\maketitle

\MSC 
22E50 
\EndMSC

\KEY 
Modular representations of $p$-adic reductive groups, 
Types, Contragredient, Intertwining
\EndKEY

\section{Introduction}

\subsection{}

Soient $p$ un nombre premier et $\G$ un groupe réductif $p$-adique. 
Dans l'étude des repré\-sen\-tations lisses de $\G$ dans des espaces vectoriels 
complexes, la restriction aux sous-groupes ouverts et compacts joue un rôle 
important \cite{BKl,BKt}.
Dans ce contexte, on fixe un sous-groupe ouvert et compact $\K$ de $\G$ et 
une représentation lisse irréductible $\tau$ de $\K$ dans un espace vectoriel 
complexe $\W$, de dimension finie.
Pour toute représentation lisse $\pi$ de $\G$ dans un espace vectoriel 
complexe $\V$, on forme l'espace d'entrelacement $\Hom_\K(\W,\V)$ qui est 
naturellement un module à droite sur l'algèbre d'entrelacement -- ou 
algèbre de Hecke -- $\Hh(\G,\K,\W)$. 
Passant aux contragrédientes, l'espace $\Hom_\K(\W^\vee,\V^\vee)$ est 
naturellement un module à droite sur $\Hh(\G,\K,\W^\vee)$.  
Mais cette dernière algèbre est naturellement anti-isomorphe à 
$\Hh(\G,\K,\W)$.  
Dans ce cas classique, les représentations lisses de $\K$ sont semi-simples, 
et $\Hom_\K(\W^\vee,\V^\vee)$ s'identifie au dual de $\Hom_\K(\W,\V)$, comme 
module à gauche sur $\Hh(\G,\K,\W)$~: nous retrouverons plus loin ce cas 
particulier. 

\subsection{}

L'arithmétique des formes modulaires a imposé l'étude des représentations
lisses de $\G$ dans des espaces vectoriels sur un corps algébriquement clos 
$\R$, de caractéristique quelconque \cite{Vigb,MSc,Herzig}.
En l'absence de méthodes analytiques, la restriction aux sous-groupes ouverts 
et compacts est en ce cas un outil encore plus important que dans le cas 
complexe. 
On a de la même façon que plus haut des modules d'entrelacement 
$\Hom_\K(\W,\V)$ sur des algèbres de Hecke $\Hh(\G,\K,\W)$ à coefficients 
dans $\R$.

Dans le cas où la caractéristique de $\R$ n'est pas $p$, on dispose d'une 
bonne théorie des contra\-grédientes, et il est nécessaire d'examiner le 
comportement des modules $\Hom_\K(\W,\V)$ par passage aux contragrédientes. 

Comme dans le cas complexe, on dispose d'un anti-isomorphisme d'algèbres de 
$\Hh(\G,\K,\W)$ sur $\Hh(\G,\K,\W^\vee)$, de sorte que 
$\Hom_\K(\W^\vee,\V^\vee)$ est un module à gauche sur $\Hh(\G,\K,\W)$.
Par ailleurs on a un accouplement naturel 
$(\h,\psi)\mapsto\langle\h,\psi\rangle$ de 
$\Hom_\K(\W,\V)\times\Hom_\K(\W^\vee,\V^\vee)$ dans $\R$.
Cependant cet accouplement peut être dégénéré. 
On cherche dans quelles conditions il est non dégénéré, et hermitien au sens 
où $\langle\h\T,\psi\rangle=\langle\h,\T\psi\rangle$ pour $\T\in\Hh(\G,\K,\W)$. 
En particulier, nous obtenons le résultat suivant (voir la remarque 
\ref{diviseur} et le théorème \ref{dimpreml}).

\begin{prop}
Supposons que le pro-ordre de $\K$ est inversible dans $\R$. 
Alors l'accouple\-ment 
$\Hom_\K(\W,\V)\times\Hom_\K(\W^\vee,\V^\vee)\to\R$
est non dégénéré et hermitien.
\end{prop}

\subsection{}

Supposons $\R$ de caractéristique $\ell$ différente de $p$, et examinons le 
cas, inévitable en pratique, où $\ell$ divise le pro-ordre de $\K$. 
Alors les représentations lisses de $\K$ ne sont plus forcément semi-simples, 
et l'accouplement peut effectivement être dégénéré. 
Cependant une situation fréquente est celle où $\K$ possède un sous-groupe 
distingué fermé $\K_1$, de pro-ordre premier à $\ell$, 
qui vérifie les deux conditions suivantes~:
\begin{enumerate}
\item
La restriction de $\tau$ à $\K_1$ est multiple d'une représentation 
irréductible $\eta$.
\item
Le composant isotypique de $\pi$ de type $\tau$ est aussi le composant 
isotypique de type $\eta$.
\end{enumerate}
En ce cas, l'accouplement plus haut est encore non dégénéré (voir 
ci-dessous \S\ref{K1}), et si de plus $\ell$ ne divise pas la dimension de $\W$, 
il est aussi hermitien (théorème \ref{dimpreml}).

\subsection{}

Le but de cette courte note est d'étudier la situation générale.
Nous prenons un point de vue abstrait, où $\G$ est un groupe localement 
profini et $\R$ un corps commutatif quelconque. 
Dans un premier temps, nous examinons les représentations lisses de $\K$ 
dans des espaces vectoriels sur $\R$, et leur comportement par passage à la 
contragrédiente. 
Nous passons ensuite à l'étude des modules d'entrelacement. 
Notre conclusion essentielle réside dans les formules \eqref{EA} et \eqref{EB} 
du \S\ref{calculs}, 
dont la comparaison permet ensuite de donner des critères et exemples 
où l'accouplement est hermitien.

\subsection{}

Pour finir, nous donnons des applications de ce travail à la théorie des types 
d'un groupe linéaire général sur un corps local non archimédien $\F$.
Nous prouvons que les types 
simples de Bushnell-Kutzko -- et plus 
généralement leur version modulaire construite dans \cite{MSt} pour les for\-mes 
intérieures de $\GL_n(\F)$, $n\>1$ -- satisfont aux conditions requises pour que 
l'accouplement soit non dégénéré et hermitien, et ce pour toute représentation 
$\V$ de $\G$ qui est un sous-quotient d'une représentation engendrée par son composant 
isotypique de type $\tau$ (théorème \ref{thtss}).

\section{Représentations de $\K$}
\label{S2}

\subsection{}

On fixe désormais un corps commutatif $\R$, et on note $\ell$ sa 
caractéristique. 
Par re\-pré\-sen\-tation $(\pi,\V)$ d'un groupe $\G$, on entendra représentation 
dans un espace vectoriel $\V$ sur le corps $\R$.

Dans ce paragraphe \ref{S2}, on fixe un groupe profini $\K$.
Si $(\pi,\V)$ est une représentation lisse de $\K$, sa contragrédiente 
$\pi^\vee$ est la représentation naturelle de $\K$ dans le sous-espace 
$\V^\vee$ du dual $\V^*$ de $\V$ formé des formes linéaires sur $\V$ 
dont le stabilisateur dans $\K$ est ouvert.
L'homomorphisme d'évaluation $\V\otimes\V^\vee\to\R$ défini par 
$(x,\l)\mapsto\l(x)$ induit un entrelacement naturel de $\V$ vers son 
bicontragrédient $\V^{\vee\vee}$. 

Si $\V$ est de dimension finie, $\pi$ se factorise par un quotient fini de 
$\K$, de sorte que $\V$ et $\V^\vee$ ont même dimension et que l'entrelacement 
naturel de $\V$ dans $\V^{\vee\vee}$ est un isomorphisme. 
Montrons par un exemple que ces bonnes propriétés ne s'étendent pas forcément 
aux représentations de dimension infinie. 

\begin{exem}
\label{EX21}
On suppose que $\ell$ n'est pas nul, et aussi que le pro-ordre de $\K$ est divisible par 
$\ell^{\infty}$~; autrement dit, pour chaque entier $r\>1$, $\K$ possède un 
sous-groupe ouvert, qu'on peut prendre distingué, d'indice divisible par 
$\ell^r$. 
Considérons la représentation (lisse) de $\K$ dans l'espace $\V$
des fonctions localement constantes de $\K$ dans $\R$,
disons par translations à droite. 
\textit{Prouvons que $\V^\vee$ est nul.}

Soit $\l\in\V^\vee$.
Il suffit de prouver que, si $\J$ est un sous-groupe ouvert de $\K$ fixant 
$\l$, alors $\l$ est nulle sur le sous-espace vectoriel de $\V$ formé des 
fonctions invariantes à droite par $\J$. 
Grâce à l'hypothèse, il existe un sous-groupe ouvert distingué $\J'$ de $\J$ 
dont l'indice est divisible par $\ell$. 
Fixons un système $\X$ de représentants de $\J'\backslash\J$ dans $\J$. 
Pour $g$ dans $\K$, la fonction caractéristique de $g\J$ est somme des 
fonctions caractéristiques $\chi_k$ de $g\J k$, pour $k$ parcourant $\X$~; 
mais $\l(\chi_k)$ ne dépend pas de $k$, et le cardinal de $\X$ est divisible 
par $\ell$, de sorte que $\l$ s'annule sur la fonction caractéristique de 
$g\J$. 
\end{exem}

\subsection{}

Fixons une représentation lisse absolument irréductible $(\tau,\W)$ de $\K$~: 
elle est donc de dimen\-sion finie sur $\R$.
Soit $(\pi,\V)$ une représentation lisse de $\K$.
La donnée de deux entrelacements $\h\in\Hom_\K(\W,\V)$ et 
$\psi\in\Hom_\K(\W^\vee,\V^\vee)$ 
définit un entrelacement~:
\begin{equation*}
\h\otimes\psi\in\Hom_{\K\times\K}(\W\otimes\W^\vee,\V\otimes\V^\vee).
\end{equation*}
Composant avec l'évaluation $\V\otimes\V^\vee\to\R$, on obtient un 
entrelacement dans $\Hom_{\K}(\W\otimes\W^\vee,\R)$, où $\K$ agit 
sur $\W\otimes\W^\vee$ par l'action diagonale. 
Comme $\tau$ est absolument irréductible de dimension finie, l'espace 
$\Hom_{\K}(\W\otimes\W^\vee,\R)$ est de dimension $1$, de base 
l'évaluation $\W\otimes\W^\vee\to\R$.
On obtient donc un accouplement naturel 
$(\h,\psi)\mapsto\langle\h,\psi\rangle$ de 
$\Hom_\K(\W,\V)\times\Hom_\K(\W^\vee,\V^\vee)$ dans $\R$, donné par la 
formule~:
\begin{equation*}
\psi(w^\vee)(\h(w))=\langle\h,\psi\rangle\cdot w^\vee(w)
\end{equation*}
pour $w\in\W$ et $w^\vee\in\W^\vee$. 

Cet accouplement peut être dégénéré, puisque nous avons vu que $\V^\vee$ 
peut être nul sans que $\V$ le soit. 
Étudions cela plus avant. 

\begin{rema}
\label{rem1}
Comme $\W$ est de dimension finie, le passage au transposé donne un 
isomor\-phisme de $\Hom_\K(\V,\W)$ sur $\Hom_\K(\W^\vee,\V^\vee)$. 
De plus, $\W$ étant absolument irréductible, la $\R$-algèbre 
$\End_\K(\W)$ est réduite 
aux homothéties $\R\cdot{\rm id}_\W$.
En ces termes, 
l'accouplement précédent se traduit en la composition 
$\Hom_\K(\W,\V)\times\Hom_\K(\V,\W)\to\End_{\K}(\W)\simeq\R$.
\end{rema}

\subsection{}
\label{S23}

Notons $\V(\tau)$ la plus grande sous-représentation de $\V$ isotypique 
de type $\tau$.
C'est l'image de $\Hom_\K(\W,\V)\otimes\W$ dans $\V$ par l'application 
injective $\h\otimes w\mapsto\h(w)$.
En particulier, si $\V$ est isotypique de type $\tau$, alors $\V^*$ est isomorphe 
à $\Hom_\K(\W,\V)^*\otimes\W^\vee$, de sorte que $\V^\vee$ est égale à $\V^*$ 
et est isotypique de type $\tau^\vee$.
De plus, l'application $\V\to\V^{\vee\vee}$ est injective. 
En général, notons $\V_\tau$ le plus grand quotient de $\V$ isotypique 
de type $\tau$, et utilisons des notations analogues pour $\V^\vee$, $\tau^\vee$.

Les entrelacements de $\V$ dans $\W$ se factorisent par l'espace $\V_\tau$, 
ce qui donne un isomorphisme de 
$\Hom_\K(\V_\tau,\W)$ sur $\Hom_\K(\V,\W)$ et, 
par transposition, 
un isomorphisme 
de $\Hom_\K(\W^\vee,(\V_\tau)^\vee)$ sur $\Hom_\K(\W^\vee,\V^\vee)$
(voir la remarque \ref{rem1}).
Considérons alors le diagramme suivant où les flè\-ches 
verticales proviennent de la projection $\V\to\V_\tau$, celle de 
droite étant un isomorphisme~:
\begin{eqnarray*}
\Hom_\K(\W,\V)\times\Hom_\K(\W^\vee,\V^\vee)\to\R\\
\downarrow\quad\quad\quad\quad\quad\quad\uparrow\quad\quad\quad\quad\quad\quad\\
\Hom_\K(\W,\V_\tau)\times\Hom_\K(\W^\vee,(\V_\tau)^\vee)\to\R
\end{eqnarray*}
Comme $\V_\tau$ est isotypique de type $\tau$, l'accouplement de la ligne 
du bas est non dégénéré~: plus précisément, 
son noyau à droite et son noyau à gauche sont nuls, 
et il induit un iso\-morphis\-me de 
$\Hom_\K(\W^\vee,(\V_\tau)^\vee)$ sur $\Hom_\K(\W,\V_\tau)^*$~; si de plus 
$\V_\tau$ est de dimension finie, il induit un iso\-mor\-phisme de 
$\Hom_\K(\W,\V_\tau)$ sur $\Hom_\K(\W^\vee,(\V_\tau)^\vee)^*$.
Remarquant que le diagramme est com\-patible aux accouplements et que les 
entrelacements de $\W$ dans $\V$ prennent leurs valeurs dans $\V(\tau)$, 
on a obtenu le résultat suivant. 

\begin{prop}
\label{P23}
Supposons que l'application $\V(\tau)\to\V_\tau$ est un isomorphisme. 
L'accouple\-ment~:
\begin{equation}
\label{AccK}
\Hom_\K(\W,\V)\times\Hom_\K(\W^\vee,\V^\vee)\to\R
\end{equation}
induit alors d'une part un isomorphisme de 
$\Hom_\K(\W^\vee,\V^\vee)$ sur $\Hom_\K(\W,\V)^*$, 
d'autre part une injection
de $\Hom_\K(\W,\V)$ dans $\Hom_\K(\W^\vee,\V^\vee)^*$, qui est une bijection si 
$\V(\tau)$ est de dimen\-sion finie. 
\end{prop}

\begin{rema}
\label{diviseur}
Si le pro-ordre de $\K$ est inversible dans $\R$, toute représentation lisse
de $\K$ est semi-simple et, quel que soit $\V$, l'hypothèse de la proposition 
est satisfaite. 
C'est le cas en particulier si $\R$ est de caractéristique nulle, par exemple 
si $\R=\mathbb{C}$.
\end{rema}

\subsection{}
\label{K1}

Dans la pratique, le pro-ordre de $\K$ n'est pas forcément inversible dans
$\R$.
Mais il arrive (voir \cite{MSt}) que $\K$ possède un sous-groupe ouvert
distingué $\K_1$ ayant les propriétés suivantes~: 
\begin{enumerate}
\item
Le pro-ordre de $\K_1$ est inversible dans $\R$.
\item
La restriction de $\tau$ à $\K_1$ est multiple d'une représentation absolument 
irréductible $\eta$ de $\K_1$.
\item
$\V(\tau)=\V(\eta)$. 
\end{enumerate}

\begin{prop}
Sous les conditions précédentes, l'application $\V(\tau)\to\V_\tau$ est un 
isomorphisme. 
\end{prop}

\begin{proof}
Comme $\eta$ est la seule représentation irréductible de $\K_1$ intervenant 
dans $\W$, sa classe d'isomorphisme est invariante par conjugaison par $\K$ et 
$\V_{\eta}$ est une représentation de $\K$.
La projection de $\V(\tau)=\V(\eta)$ dans $\V_{\eta}$ est $\K$-équivariante 
et, comme le pro-ordre de $\K_1$ est inversible dans $\R$, c'est même un isomorphisme. 
En particulier $\V_{\eta}$ est isotypique de type $\tau$ et il s'ensuit que 
$\V_\tau=\V_{\eta}$.
\end{proof}

\subsection{}
\label{par25}

Revenons à notre situation générale, et 
donnons des conditions nécessaires et suffisantes 
pour que $\V(\tau)\to\V_\tau$ soit un iso\-mor\-phis\-me. 

\begin{prop}
\label{WSS}
Les conditions suivantes sont équivalentes~:
\begin{enumerate}
\item
L'application $\V(\tau)\to\V_\tau$ est un isomorphisme.
\item 
L'accouplement \eqref{AccK} induit un isomorphisme entre $\Hom_\K(\V,\W)$ et 
$\Hom_\K(\W,\V)^*$.  
\item
$\V$ se décompose sous la forme $\V(\tau)\oplus\V'$
avec 
$\Hom_\K(\W,\V')=\Hom_\K(\V',\W)=\{0\}$. 
\end{enumerate}
\end{prop}

\begin{proof}
On a déjà prouvé que la condition 1 implique la condition 2
(voir la proposition \ref{P23} et la remarque \ref{rem1}) et on 
vérifie immédiatement que la condition 3 implique la condi\-tion 1. 
Notons $\V'$ le noyau de $\V\to\V_\tau$.
Le noyau à droite de \eqref{AccK} est formé des 
$\psi\in\Hom_\K(\V,\W)$ qui sont nuls sur $\V(\tau)+\V'$ et son 
noyau à gauche est formé des $\h\in\Hom_\K(\W,\V)$ dont l'image 
est incluse dans $\V(\tau)\cap\V'$. 
Si la condition $2$ est vérifiée, on obtient donc $\V(\tau)\cap\V'=\{0\}$ 
et~:
\begin{equation*}
\Hom_\K(\V/(\V(\tau)+\V'),\W)=\{0\}.
\end{equation*}
Comme le quotient $\V/(\V(\tau)+\V')$ est isotypique de type $\tau$, 
on en déduit que $\V=\V(\tau)\oplus\V'$. 
On a aussi $\V'(\tau)=\V(\tau)\cap\V'=\{0\}$ et 
$\V'_\tau=\V/(\V(\tau)+\V') =\{0\}$,
ce qui termine la preuve. 
\end{proof}

\begin{defi}
La représentation $\V$ est dite $\W$-\textit{semi-simple} si les conditions 
équivalentes de la proposition \ref{WSS} sont vérifiées. 
\end{defi}

La propriété de $\W$-semi-simplicité n'est stable ni par passage à un 
sous-objet, ni par passage à un quotient, comme le montre l'exemple ci-dessous. 

\begin{exem}
\label{EX1}
On suppose que $\K$ est le groupe $\GL_2(k)$, 
où $k$ est un corps fini dont le 
cardinal $q$ est d'ordre $2$ dans $\mult\R$. 
On considère la représentation de $\K$, par translations à droite, 
sur l'espace $\E$ des fonctions de $\K$ dans $\R$ qui sont invariantes 
à gauche par le sous-groupe des ma\-tri\-ces triangulaires supérieures. 
Elle est indécomposable et de longueur $3$~: elle admet une unique suite 
de composition $\{0\}\subseteq\X\subseteq\Y\subseteq\E$, où $\X$ et 
$\E/\Y$ sont isomorphes au caractère trivial de $\K$ et où $\Y/\X$ est une 
représentation irréductible cuspidale de $\K$, notée $(\tau,\W)$.
Remarquons que $\X$ est le sous-espace des fonctions constantes. 

Le quotient $\V=\E/\X$ est indécomposable et de longueur $2$. 
L'espace $\V(\tau)$ est non nul, iso\-mor\-phe à $\tau$,
mais on va voir que $\V_\tau$ est réduit à $\{0\}$. 
Si ce n'était pas le cas, alors $\tau^\vee$ apparaîtrait 
comme sous-représentation de $\V^\vee$.
Or $\V^\vee$ est indécomposable, et son unique sous-représentation 
irréductible est le caractère trivial de $\K$. 

La représen\-ta\-tion $\E$ est $\W$-semi-simple, avec $\E(\tau)=\{0\}$, 
mais ni son unique quotient $\V$ de longueur $2$ ni son unique 
sous-représentation $\Y$ de longueur $2$ ne le sont.
\end{exem}

Ceci nous conduira à introduire 
dans le para\-graphe suivant une propriété plus forte. 
Signalons toutefois les propriétés suivantes. 

\begin{prop}
\begin{enumerate}
\item
On suppose que l'application $\V(\tau)\to\V_\tau$ est surjective. 
Alors pour tout quotient $\Q$ de $\V$, l'application $\Q(\tau)\to\Q_\tau$ est 
surjective.  
\item
On suppose que l'application $\V(\tau)\to\V_\tau$ est injective. 
Alors pour toute sous-repré\-senta\-tion $\X$ de $\V$, l'application 
$\X(\tau)\to\X_\tau$ est injective.  
\end{enumerate}
\end{prop}

\begin{proof}
Soit $\X$ un sous-espace de $\V$ stable par $\K$, et soit $\Q$ le quotient 
de $\V$ par $\X$. 
On prouve d'abord (1).
Le quotient~:
\begin{equation*}
(\V(\tau)+\X)/\X\simeq\V(\tau)/(\X\cap\V(\tau))
\end{equation*}
est isotypique de type $\tau$~; c'est donc un sous-espace de $\Q(\tau)$. 
La composée de $\V\to\Q$ et $\Q\to\Q_\tau$ se factorise par $\V\to\V_\tau$, \ie que 
$\Q_\tau$ est un quotient de $\V_\tau$. 
Il suffit donc de prouver que l'application
$(\V(\tau)+\X)/\X\to\V_\tau$ est surjective. 
On note $\V'$ le noyau de $\V\to\V_\tau$. 
Alors par hypothèse faite sur $\V(\tau)\to\V_\tau$, 
on a $\V=\V(\tau)+\V'$, ce qui prouve la 
surjectivité voulue. 

On prouve maintenant (2). 
On note toujours $\V'$ le noyau de $\V\to\V_\tau$. 
Le quotient~:
\begin{equation*}
\X/(\X\cap\V')\simeq
(\X+\V')/\V'
\end{equation*}
est isotypique de type $\tau$~; c'est donc un quotient de $\X_\tau$.
Par hypothèse faite sur $\V(\tau)\to\V_\tau$, 
on a $\V(\tau)\cap\V'=\{0\}$.
Ainsi $\X(\tau)=\X\cap\V(\tau)$ se plonge 
dans $\X/(\X\cap\V')$, ce dont on 
déduit que l'application $\X(\tau)\to\X_\tau$ est injective.  
\end{proof}

\begin{rema}
\label{remaoplusWss}
Soit $(\V_i)_{i\in\I}$ une famille de représentations de $\K$, 
et soit $\V$ la somme directe des $\V_i$, $i\in\I$.
Les conditions suivantes sont équivalentes~:
\begin{enumerate}
\item 
Pour tout $i\in\I$, l'application $\V_i(\tau)\to\V_{i,\tau}$ est surjective 
(resp. injective).
\item
L'application $\V(\tau)\to\V_\tau$ est surjective (resp. injective).
\end{enumerate}
\end{rema}

\subsection{}
\label{fWss}

Introduisons dans ce paragraphe une propriété plus stable que celle 
de $\W$-semi-simplicité.

\begin{defi}
Une représentation lisse $\V$ de $\K$ est dite 
\emph{fortement $\W$-semi-simple} si elle 
se décompose sous la forme $\V(\tau)\oplus\V'$, 
où $\V'$ ne possède aucun sous-quotient isomorphe à $\W$. 
\end{defi}

Bien sûr, une représentation fortement $\W$-semi-simple
est $\W$-semi-simple (proposition \ref{WSS}).

\begin{prop}
\label{fWsssqs}
Tout sous-quotient d'une représentation fortement $\W$-semi-simple 
de $\K$ est fortement $\W$-semi-simple. 
\end{prop}

\begin{proof}
Soit $\E$ une représentation lisse fortement $\W$-semi-simple de $\K$, 
qu'on écrit sous la forme $\E(\tau)\oplus\E'$. 

\begin{lemm}
Soit $\V$ une sous-représentation de $\E$.  
Alors $\V=(\V\cap\E(\tau))\oplus(\V\cap\E')$.
\end{lemm}

\begin{proof}
On note 
$\U$ le quotient de $\V$ par $(\V\cap\E(\tau))\oplus(\V\cap\E')$.
On va prouver que $\U$ est nul. 
D'une part, $\U$ est un quotient de~:
\begin{equation*}
\V/(\V\cap\E')\simeq(\V+\E')/\E'\subseteq\E/\E'\simeq\E(\tau)
\end{equation*}
qui est isotypique de type $\tau$.
D'autre part, $\U$ est un quotient de~:
\begin{equation*}
\V/(\V\cap\E(\tau))\simeq(\V+\E(\tau))/\E(\tau)\subseteq\E/\E(\tau)\simeq\E'
\end{equation*}
qui ne contient aucun sous-quotient irréductible isomorphe à $\tau$. 
On en déduit que $\U$ est nul.
En particulier $\V$ est fortement $\W$-semi-simple, 
avec $\V(\tau)=\V\cap\E(\tau)$ et $\V'=\V\cap\E'$. 
\end{proof}

Soit maintenant $\X$ le quotient de $\E$ par la sous-représentation $\V$. 
D'après le lemme précédent, on voit que $\X=\X(\tau)\oplus\X'$ avec 
$\X(\tau)=\E(\tau)/(\V\cap\E(\tau))$ et $\X'=\E'/(\V\cap\E')$, ce qui prouve que 
$\X$ est fortement $\W$-semi-simple.
\end{proof}

La propriété de $\W$-semi-simplicité forte est stable par sommes directes 
arbitraires~: si $(\V_i)_{i\in\I}$ est une famille de représentations de $\K$, 
et si $\V$ est la somme directe des $\V_i$, $i\in\I$, alors $\V$ est fortement 
$\W$-semi-simple si et seulement si chacune des $\V_i$, $i\in\I$, est fortement 
$\W$-semi-simple.

En revanche, cette propriété n'est pas stable par extensions~: toute représentation 
irréductible de $\K$ est fortement $\W$-semi-simple, mais une extension 
non triviale de $\tau$ par une représentation irréductible non isomorphe à 
$\tau$ (s'il en existe) ne l'est pas. 

\subsection{}

Étudions le comportement de la notion de $\W$-semi-simplicité forte 
par passage à la contragrédiente.

Pour toute représentation lisse $\E$ de $\K$
et pour toute partie $\SS$ de $\E^\vee$, on notera $\SS^\circ$ l'orthogonal 
de $\SS$ dans $\E$, \ie l'intersection des noyaux des éléments de $\SS$. 

\begin{prop}
\label{contrafWss1}
Soit $\E$ une représentation lisse fortement $\W$-semi-simple de $\K$. 
Alors sa contragrédiente $\E^\vee$ est fortement $\W^\vee$-semi-simple. 
\end{prop}

\begin{proof}
Soit $\E$ une représentation fortement $\W$-semi-simple de 
$\K$ qu'on décompose sous la forme $\E(\tau)\oplus\E'$ où $\E'$ 
est une sous-représentation de $\E$ ne 
possèdant aucun sous-quotient isomorphe à $\W$.  
Sa contragrédiente $\E^\vee$ s'écrit $\E(\tau)^*\oplus\E'^\vee$ et on va 
montrer que $\E'^\vee$ ne possède aucun sous-quotient isomorphe à $\W^\vee$.  
On a le lemme immédiat suivant. 

\begin{lemm}
\label{Haller}
Soient $\T,\U$ des sous-espaces de $\E^\vee$ 
tels que $\T\subseteq\U$.  
Alors l'application qui à $t\in\T^\circ$ associe la forme linéaire 
$\xi\mapsto\xi(t)$ sur $\U$ induit 
un homomorphisme injectif de $\T^\circ/\U^\circ$ dans $(\U/\T)^\vee$. 
En outre, si $\U$ est de dimension finie, c'est un isomorphisme.
\end{lemm}

Supposons qu'il y a des sous-espaces $\T\subseteq\U$ de $\E'^\vee$ stables
par $\K$ et tels que le quotient $\U/\T$ est isomorphe à $\W^\vee$. 
Quitte à remplacer $\U$ par la sous-représentation de $\U$ engendrée par un vecteur 
qui est dans $\U$ mais pas dans $\T$, on peut supposer que $\U$ est de 
dimension finie. 
On a ainsi un isomorphisme
$\T^\circ/\U^\circ\simeq(\U/\T)^\vee$, ce qui contredit l'hypothèse faite sur 
$\E'$. 
\end{proof}

La réciproque n'est pas vraie en général~: si $\V$ est la représentation de 
l'exemple \ref{EX21}, alors sa contragrédiente, qui est nulle, 
est fortement $\W$-semi-simple pour n'importe quel $\W$,
tandis que $\V$ n'est fortement $\W$-semi-simple pour aucun $\W$. 
On a toutefois le résultat suivant. 

\begin{prop}
Supposons que $\K$ a un sous-groupe ouvert $\K_1$ de 
pro-ordre inversible dans $\R$. 
Alors une représentation de $\K$ est fortement $\W$-semi-simple si et 
seulement si sa contragrédiente est fortement $\W^\vee$-semi-simple.
\end{prop}

\begin{proof}
La proposition \ref{contrafWss1} prouve l'une des implications. 
Par hypothèse sur $\K$, $\E$ se plonge naturellement dans son bicontragrédient 
$\E^{\vee\vee}$. 
L'implication réciproque se déduit des propositions \ref{fWsssqs} et 
\ref{contrafWss1}. 
\end{proof}

\subsection{}

Revenons finalement 
à la notion de $\W$-semi-simplicité et à la façon dont cette 
propriété se comporte par passage à la contragrédiente.

\begin{prop}
Supposons que $\K$ a un sous-groupe ouvert de 
pro-ordre inversible dans $\R$. 
Alors une représentation de $\K$ est $\W$-semi-simple si et seulement 
si sa contragrédiente est $\W^\vee$-semi-simple.
\end{prop}

\begin{proof}
Étant donnée une représentation lisse $\E$ de $\K$ qui n'a $\W$ ni 
pour sous-repré\-sen\-ta\-tion ni pour quotient, 
il s'agit de montrer que $\W^\vee$ n'est ni sous-représentation 
ni quotient de $\E^\vee$. 
Supposons que $\E^\vee$ a une sous-représentation $\U$ isomorphe 
à $\W^\vee$.
D'après le lemme \ref{Haller} avec $\T=\{0\}$, 
on obtient un isomorphisme $\E/\U^\circ\simeq\U^\vee$, ce qui 
contre\-dit le fait que $\E$ n'a pas de quotient isomorphe à $\W$. 
Si $\E$ est de dimension finie sur $\R$, on peut encore raisonner 
de façon analogue pour prouver que $\E^\vee$ n'a pas de quotient isomorphe 
à $\W^\vee$.

Il reste à prouver que $\E^\vee$ n'a pas de quotient isomorphe à $\W^\vee$ 
dans le cas où $\E$ n'est pas de dimension finie sur $\R$.
Puisque $\W$ est lisse et de dimension finie, on peut supposer que 
$\K$ possède un sous-groupe ouvert distingué $\K_1$ de pro-ordre 
inversible dans $\R$, et que $\K_1$ agit trivialement sur $\W$. 
Écrivons $\E=\E^{\K_1}\oplus\E_1$ où $\E^{\K_1}$ désigne l'espace des 
vecteurs de $\E$ invariants par $\K_1$ et où $\E_1$ n'a aucun sous-quotient 
isomorphe à $\W$. 
Quitte à remplacer $\K$ par $\K/\K_1$ et $\E$ par $\E^{\K_1}$, 
on peut supposer que $\K$ est un groupe fini. 

L'algèbre de groupe $\A=\R[\K]$ est de dimension finie sur $\R$, 
en particulier elle est artinienne. 
On note $\rr$ son radical. 
D'après le corollaire du §10, n°2 de \cite{BourbakiALG}, 
pour tout $\A$-module $\M$, le radical de $\M$ 
(\ie le plus petit sous-module $\N$ tel que $\M/\N$ soit semi-simple) 
est égal à $\rr\M$.
On note $\M[\rr]$ le plus grand sous-module 
semi-simple de $\M$.  

\begin{lemm}
\label{l27}
Soit $\M$ un $\A$-module.
Alors $\rr\M^*$ est égal à l'orthogonal 
de $\M[\rr]$ dans $\M^*$.  
\end{lemm}

\begin{proof}
D'abord, la restriction des formes linéaires de $\M$ à $\M[\rr]$ induit 
un isomorphis\-me de $\M^*/\M[\rr]^\circ$ sur $\M[\rr]^*$.
Ce dernier est semi-simple, ce qui implique que $\rr\M^*\subseteq\M[\rr]^\circ$. 
En outre, si $\M$ est de dimension finie sur $\R$, cette inclusion est 
une égalité. 

Écrivons $\M$ comme la réunion de sous-modules $\M_i$ 
de dimension finie sur $\R$, pour $i\in\I$. 
Alors $\M^*$ est la limite projective des $\M_i^*$ 
(si $\M_i^{}\subseteq\M_j^{}$, on a un morphisme surjectif de $\M_j^*$ dans 
$\M_i^*$ défini par restriction des formes linéaires) et $\M[\rr]^\circ$ est 
la limite projective des $\M^{}_i[\rr]^\circ$. 
Pour tout $i\in\I$, on a le dia\-gram\-me commutatif suivant~:
\begin{equation*}
\begin{matrix}
\rr\M^*&\subseteq&\M[\rr]^\circ\\
\downarrow&&\downarrow\\
\rr\M^*_i&=&\M_i^{}[\rr]^\circ
\end{matrix}
\end{equation*}
et il s'agit donc de prouver que $\rr\M^*$ est la limite projective des 
$\rr\M_i^*$. 
Plus précisément, on a un homomorphisme 
injectif $\xi$ de $\rr\M^*$ vers la li\-mi\-te projective des $\rr\M_i^*$, et 
il s'agit de montrer que $\xi$ est surjectif. 

Choisissons une base $(a_1,\dots,a_k)$ de $\rr$ comme espace vectoriel 
sur $\R$. 
Pour chaque $i\in\I$, on a une application linéaire $f^{}_i$ de $(\M_i^*)^k$ dans 
$\M_i^*$ définie par~:
\begin{equation*}
f_i(m_1,\dots,m_k)=a_1m_1+\dots+a_km_k,
\end{equation*}
d'image $\rr\M_i^*$.
De façon analogue, on a une application linéaire $f$ de $(\M^*)^k$ dans 
$\M^*$ dont l'image est $\rr\M^*$.
Soit $x\in\M^*$ et soit $x_i$ son image dans $\M_i^*$.
On suppose que $x_i\in\rr\M_i^*$ pour chaque $i\in\I$. 
Alors $f_i^{-1}(x^{}_i)$ est un sous-espace affine non vide de $(\M_i^*)^k$ 
de direction $\Ker(f_i)$.
D'après le théorème 1 de \cite{BourbakiENS}, chapitre III, §7, n°4, la limite projective 
des $f_i^{-1}(x^{}_i)$ est non vide, \ie que l'équation $f(y)=x$ admet une 
solution dans $(\M^*)^k$. 
\end{proof}

Si l'on applique le lemme à $\M=\E$, on obtient un isomorphisme de 
$\E^*/\rr\E^*=\E^*/\E[\rr]^\circ$ dans $\E[\rr]^*$. 
Comme $\W$ n'apparaît pas dans $\E[\rr]$, on en déduit 
que $\W^\vee$ n'apparaît pas comme quotient de $\E^*$.
\end{proof}

\section{Modules d'entrelacement}

\subsection{}

On suppose maintenant que $\G$ est un groupe localement profini et que $\K$ 
est un sous-groupe ouvert et compact de $\G$.
On fixe comme précédemment une représentation lisse absolument irréductible 
$(\tau,\W)$ de $\K$.
On note $\Hh=\Hh(\G,\K,\W)$ l'algèbre d'entrelacement -- ou algèbre de Hecke -- 
de $\tau$ dans $\G$~: on la voit comme formée des fonctions $\T$ de $\G$ 
dans $\End_\R(\W)$, à support compact, telles que~:
\begin{equation*}
\T(hgk)=\tau(h)\circ\T(g)\circ\tau(k)
\end{equation*}
pour $g\in\G$ et $h,k\in\K$. 
Cette formule signifie que $\T(g)$ entrelace $\tau$ avec $\tau^{g}$.
La loi de composition est la loi de convolution~:
\begin{equation*}
\T*\SS(g)=\sum\limits_{\K\backslash \G} \T(gx^{-1})\circ\SS(x)
\end{equation*}
où la somme porte sur un système quelconque de représentants $x$ de 
$\K\backslash\G$ dans $\G$.

On note $\Hh^\vee=\Hh(\G,\K,\W^\vee)$ l'algèbre de Hecke de $\tau^\vee$ dans 
$\G$.
L'application qui à $\T$ associe $\T^\vee:g\mapsto\T(g^{-1})^\vee$ 
(où $\T(g^{-1})^\vee$ désigne l'endomorphisme transposé de $\T(g^{-1})$) est, 
comme on le vérifie aussitôt, un anti-iso\-morphisme de $\R$-algèbres de $\Hh$ 
sur $\Hh^\vee$.  

\subsection{}
\label{calculs}

Soit $(\pi,\V)$ une représentation lisse de $\G$.
L'espace d'entrelacement $\Hom_\K(\W,\V)$ est alors un $\Hh$-module à droite~: 
pour $\h\in\Hom_\K(\W,\V)$ et $\T\in\Hh$, on a~:
\begin{equation*}
\h \T=\sum\limits_{\K\backslash \G} \pi(x^{-1})\circ\h\circ \T(x),
\end{equation*}
la somme portant sur un système quelconque de représentants $x$ 
de $\K\backslash \G$ dans $\G$.
De la même façon, $\Hom_\K(\W^\vee,\V^\vee)$ est un $\Hh^\vee$-module 
à droite.
C'est également, \textit{via} l'anti-isomorphisme $\T\mapsto\T^\vee$, un 
$\Hh$-module à gauche, et il est naturel de se demander si 
l'accouplement~: 
\begin{equation}
\label{accGH}
\Hom_\K(\W,\V)\times\Hom_\K(\W^\vee,\V^\vee)\to\R
\end{equation}
est hermitien, au sens où~:
\begin{equation*}
\label{autodual}
\langle\phi\T,\psi\rangle=
\langle\phi,\T\psi\rangle
\end{equation*}
pour tous $\phi\in\Hom_\K(\W,\V)$, $\psi\in\Hom_\K(\W^\vee,\V^\vee)$ 
et $\T\in\Hh$, où $\T\psi$ désigne bien sûr $\psi\T^\vee$.

Il sera commode d'utiliser la remarque \ref{rem1} et de considérer plutôt 
l'accouplement~:
\begin{equation}
\label{accVS}
\Hom_\K(\W,\V)\times\Hom_\K(\V,\W)\to\R.
\end{equation}
L'action à gauche de $\Hh$ sur $\Hom_\K(\V,\W)$ est alors donnée par~:
\begin{equation}
\label{actGH}
\T\psi=\sum\limits_{\K\backslash \G} \T(x^{-1})\circ\psi\circ\pi(x)
\end{equation}
pour $\psi\in\Hom_\K(\V,\W)$ et $\T\in\Hh$, 
la somme portant sur un système quelconque de représentants de 
$\K\backslash \G$ dans $\G$.

\begin{rema}
Comme l'induction lisse de $\K$ à $\G$ est adjointe à droite 
de la restriction de $\G$ à $\K$, l'espace $\Hom_{\K}(\V,\W)$ est 
naturellement un module à gauche sur l'algèbre des 
endomor\-phismes de l'induite lisse $\Ind^\G_\K(\tau)$, notée $\HH$.  
Le produit de convolution~:
\begin{equation*}
\T*f:g\mapsto\sum\limits_{\K\backslash \G} \T(x^{-1})(f(xg))
\end{equation*} 
pour $\T\in\Hh$, $f\in\Ind^\G_\K(\tau)$ et $g\in\G$,
la somme portant sur un système quelconque de représentants de 
$\K\backslash \G$ dans $\G$, 
fait de cette induite lisse un module à gauche sur $\Hh$ 
(on vérifie immédiatement que $(\T*\SS)*f=\T*(\SS*f)$ pour tous 
$\T,\SS\in\Hh$).
Ceci définit un morphisme 
de $\R$-algèbres $i$ de $\Hh$ dans $\HH$.
La restriction du $\HH$-module $\Hom_{\K}(\V,\W)$ à $\Hh$ induite 
par le morphisme $i$ coïncide avec la structure de $\Hh$-module à gauche sur 
$\Hom_{\K}(\V,\W)$ définie par \eqref{actGH}.
En effet, l'action de $\HH$ sur $\Hom_{\K}(\V,\W)$ est donnée par~:
\begin{equation*}
a\psi:v\mapsto a(\widetilde{\psi}(v))(1),
\end{equation*}
pour $a\in\HH$, $\psi\in\Hom_{\K}(\V,\W)$, $v\in\V$, 
où $\widetilde{\psi}(v)$ 
est l'élément de $\Ind^\G_\K(\tau)$ défini par 
$g\mapsto\psi(\pi(g)v)$.
Par conséquent, pour $\T\in\Hh$, on a~:
\begin{equation*}
i(\T)\psi:v\mapsto i(\T)(\widetilde{\psi}(v))(1)
=(\T*\widetilde{\psi}(v))(1)
=\sum\limits_{\K\backslash \G} \T(x^{-1})(\psi(\pi(x)v)
\end{equation*}
\ie que $i(\T)\psi$ est égal à $\T\psi$ d'après la formule \eqref{actGH}. 
\end{rema}

Il s'agit donc de comparer $\psi\circ(\h\T)$ et $(\T\psi)\circ\phi$ pour 
$\phi\in\Hom_\K(\W,\V)$ et $\psi\in\Hom_\K(\V,\W)$. 
Fixons de tels $\h$ et $\psi$.

\begin{lemm}
\label{Intertwining}
\begin{enumerate}
\item 
Soit $g\in\G$. 
L'opéra\-teur~:
\begin{equation*}
\E(\h,\psi,g)=\psi\circ\pi(g)\circ\h.
\end{equation*} 
entrelace $\tau$ avec $\tau^g$. 
\item
L'application $\E:g\mapsto\E(\phi,\psi,g)$ de $\G$ dans
$\End_\R(\W)$ vérifie~:
\begin{equation}
\label{condHT}
\E(kgk')=\tau(k)\circ\E(g)\circ\tau(k')
\end{equation}
pour tous $k,k'\in\K$ et $g\in\G$.
\end{enumerate}
\end{lemm}

\begin{rema}
La restriction de $\E$ à chaque double classe $\K g\K$ de $\G$ est dans $\Hh$.
\end{rema}

\begin{proof}
On pose $\E(g)=\E(\h,\psi,g)$.
Pour tout $k\in \K\cap g^{-1}\K g$, on a~:
\begin{eqnarray*}
\tau(gkg^{-1})\circ\E(g)&=&\psi\circ\pi(gkg^{-1})\circ\pi(g)\circ\h\\
&=&\psi\circ\pi(g)\circ\pi(k)\circ\h\\
&=&\E(g)\circ\tau(k),
\end{eqnarray*}
ce qui prouve le point 1.
Le point 2 se déduit d'un calcul immédiat. 
\end{proof}

Soit $\T\in\Hh$.
On a les formules~:
\begin{eqnarray*}
\psi\circ\h\T&=&\sum\limits_{\K\backslash \G} \E(\phi,\psi,x^{-1})\circ \T(x),\\
\T\psi\circ\phi&=&\sum\limits_{\K\backslash \G} \T(x^{-1})\circ \E(\phi,\psi,x)
\end{eqnarray*}
dans $\End_\K(\W)$, la somme portant sur un système quelconque de 
représentants de $\K\backslash \G$ dans $\G$.
Ces deux quantités sont des homothéties de $\W$, \ie des endomorphismes 
de $\W$ de la forme $\l\cdot{\rm id}_\W$ avec $\l\in\mult\R$.

Pour comparer ces deux quantités, il 
suffit de le faire pour $\T$ supporté par une seule double classe $\K g\K$. 
Dans ce cas, les formules ci-dessus deviennent~:
\begin{eqnarray}
\label{EA}
\psi\circ\h \T&=&
\sum\limits_{(\K\cap\K^g)\backslash \K}\tau(x^{-1})\circ \E(\phi,\psi,g^{-1})\circ \T(g)\circ\tau(x),\\
\label{EB}
\T\psi\circ\phi&=&
\sum\limits_{(\K\cap\K^{g^{-1}})\backslash \K}\tau(x^{-1})\circ \T(g)\circ \E(\phi,\psi,g^{-1})\circ\tau(x)
\end{eqnarray}
dans $\End_\K(\W)$, où l'on a noté $\K^g=g^{-1}\K g$ pour $g\in\G$ et où 
les sommes portent sur des systè\-mes de représentants des ensembles-quotients 
indiqués. 

\begin{theo}
\label{dimpreml}
Si la dimension de $\W$ est première à l'exposant caractéristique de 
$\R$, alors l'accouplement \eqref{accGH}
est hermitien quelle que soit la représentation lisse $(\pi,\V)$ de $\G$. 
\end{theo}

\begin{proof}
En effet, les termes des deux sommes \eqref{EA} et \eqref{EB} ont la même 
trace, et les deux sommes ont le même nombre de termes, à savoir le nombre de 
classes de $\K$ dans $\K g\K$, à droite ou à gauche. 
Comme les deux sommes sont des homothéties de $\W$, et comme 
la dimension de $\W$ est première à l'exposant caractéristique de $\R$, 
le résultat s'ensuit. 
\end{proof}

\subsection{}

Le théorème précédent appelle un certain nombre de remarques. 
\begin{enumerate}
\item
Nous n'avons pas fait d'hypothèse sur la non-dégénérescence de 
l'accouplement, qui ne semble guère reliée avec le fait qu'il soit hermitien. 
\item
Nous ne donnons pas d'exemple où l'accouplement ne soit pas hermitien. 
Si l'on revient au lemme \ref{Intertwining}, il paraît peu 
vraisemblable qu'on ait toujours $\E(g^{-1})\circ\T(g)=\T(g)\circ\E(g^{-1})$. 
Cependant, une situation fréquente en théorie des types \cite{BKl,MSt} est celle 
où~: 
\begin{equation*}
\dim\Hom_{\K\cap\K^g}(\tau,\tau^g)\<1
\end{equation*}
pour tout $g\in\G$.
Alors $\E(g^{-1})$ et $\T(g)$ sont pris chacun dans un espace vectoriel de 
rang $1$, et on peut penser qu'il y a plus de chances qu'ils commutent~: 
voir la section \ref{S4}. 
\item
Supposons que $(\pi,\V)$ est telle que $\V(\tau)$ et $\V_\tau$ soient 
irréductibles. 
Alors ou bien l'appli\-cation $\V(\tau)\to\V_\tau$ est nulle et l'accouplement 
est nul également, donc hermitien.
Ou bien cette application est bijective, 
et l'accouplement est non dégénéré~;
alors $\Hom_\K(\W,\V)$ et $\Hom_\K(\V,\W)$ sont de dimension $1$, 
et pour prouver que l'accouplement est hermitien, il suffit de prouver que 
les quantités 
\eqref{EA} et \eqref{EB} sont égales quand $\h$ et $\psi$ sont des bases de 
ces espaces, telles qu'on ait $\psi\circ\h={\rm id}_\W$. 
En ce cas, il y a des caractères $\A$ et $\B$ de l'algèbre $\Hh$ tels que~:
\begin{equation*}
\psi\circ\h\T=\A(\T)\cdot{\rm id}_\W,
\quad
\T\psi\circ\h=\B(\T)\cdot{\rm id}_\W,
\end{equation*}
pour tout $\T\in\Hh$, et il suffit de prouver $\A(\T)=\B(\T)$ pour $\T$
parcourant un système générateur de cette algèbre. 
\end{enumerate}

\subsection{}
\label{hyp12enonce}

Notons $\HT$ l'espace des fonctions de $\G$ dans $\End_\R(\W)$ 
vérifiant la condition \eqref{condHT}.
Identifions $\Hh$ au sous-espace 
de $\HT$ formé des fonctions à support compact. 
Alors $\HT$ est naturellement un module à droite et à gauche sur $\Hh$~:
pour $f\in\HT$ et $\T\in\Hh$, on a~:
\begin{eqnarray*}
f*\T:g&\mapsto&\sum\limits_{\K\backslash \G} f(gx^{-1})\circ\T(x),\\
\T*f:g&\mapsto&\sum\limits_{\K\backslash \G}\T(gx^{-1})\circ f(x),
\end{eqnarray*}
la somme portant sur un système quelconque de représentants de 
$\K\backslash \G$ dans $\G$ (on vérifie im\-mé\-diatement que l'on a 
$f*(\T*\SS)=(f*\T)*\SS$ et $(\T*\SS)*f=\T*(\SS*f)$ pour tout $f\in\HT$ 
et tous $\T,\SS\in\Hh$).

\begin{lemm}
\label{dg1}
Soit $\T\in\Hh$ supporté par une seule double classe $\K g\K$, 
et soient $f,f'\in\HT$. 
On suppose que $f(g^{-1})=f'(g^{-1})$. 
Alors $f*\T(1)=f'*\T(1)$ et $\T* f(1)=\T* f'(1)$. 
\end{lemm}

\begin{proof}
De façon analogue à \eqref{EA} et \eqref{EB}, on a~: 
\begin{eqnarray*}
f*\T(1)&=&
\sum\limits_{(\K\cap\K^g)\backslash \K}
\tau(x^{-1})\circ f(g^{-1})\circ\T(g)\circ\tau(x),\\
\T*f(1)&=&
\sum\limits_{(\K\cap\K^{g^{-1}})\backslash \K}
\tau(x^{-1})\circ\T(g)\circ f(g^{-1})\circ\tau(x),\\
\end{eqnarray*}
ce qui prouve le résultat. 
\end{proof}

On forme les hypothèses suivantes. 

\begin{hypo}
\label{h3}
Pour tout $g\in \G$, l'espace d'entrelacement
$\Hom_{\K\cap \K^g}(\tau,\tau^g)$ est de dimension au plus $1$.
\end{hypo}

\begin{hypo}
\label{h4}
Tout élément non nul de $\Hh$ supporté par une seule double classe est inversible. 
\end{hypo}

Ces hypothèses sont vérifiées fréquemment en théorie des types 
(\cite{MSt}, lem\-me 2.19 et co\-rol\-laire 2.23).

\begin{prop}
\label{hermitien}
Sous les hypothèses \ref{h3} et \ref{h4}, l'accouplement \eqref{accVS} 
est hermitien.
\end{prop}

\begin{proof}
Soient $\phi\in\Hom_\K(\W,\V)$, $\psi\in\Hom_\K(\V,\W)$ et $\T\in\Hh$. 
On suppose que $\T$ est supporté par une seule double classe $\K g\K$. 
D'après l'hypothèse \ref{h4}, $\T$ est inversible dans $\Hh$.
Ainsi $\T^{-1}(g^{-1})$ entrelace $\tau$ avec $\tau^{g^{-1}}$.

\begin{lemm}
\label{sym}
$\T^{-1}(g^{-1})$ est non nul. 
\end{lemm}

\begin{proof}
Par définition du produit de convolution, on a~:
\begin{equation*}
{\rm id}_\W=
\T^{-1}*\T(1)=\sum\limits_{(\K\cap\K^g)\backslash \K}
\tau(x^{-1})\circ\T^{-1}(g^{-1})\circ \T(g)\circ\tau(x),
\end{equation*}
où la somme porte sur un système quelconque de représentants de 
$\K\cap\K^g$ dans $\K$.
Par consé\-quent, $\T^{-1}(g^{-1})$ n'est pas nul. 
\end{proof}

D'après le lemme \ref{Intertwining}, l'opérateur $\E(\phi,\psi,g^{-1})$ 
entrelace $\tau$ avec $\tau^{g^{-1}}$. 
D'après l'hypothèse \ref{h3}, il y a donc un scalaire $z\in\R$ tel que~:
\begin{equation*}
\E(\phi,\psi,g^{-1})=z \T^{-1}(g^{-1}).
\end{equation*}
Si l'on applique le lemme \ref{dg1} avec $f=\E$ et $f'=z \T^{-1}$, on a~:
\begin{equation*}
\E*\T(1)=z \T^{-1}*\T(1)=z \T*\T^{-1}(1)=\T*\E(1),
\end{equation*}
\ie que \eqref{EA} et \eqref{EB} coïncident. 
\end{proof}

\begin{rema}
L'hypothèse \ref{h4} pourrait être remplacée par l'hypothèse plus faible 
suivante~: 
pour tout $g\in\G$ et tout $\T\in\Hh$ de support $\K g\K$, il y a une 
fonction $f\in\HT$ 
telle que $f(g^{-1})\neq0$ et $f*\T(1)=\T* f(1)$.
\end{rema}

Si l'on ajoute aux considérations ci-dessus celles des paragraphes 
\ref{par25} et \ref{fWss}, on a le théorème suivant. 

\begin{theo}
\label{Herm}
Supposons que les hypothèses \ref{h3} et \ref{h4} sont 
vérifiées et que l'induite compacte 
$\ind^\G_\K(\tau)$ est for\-te\-ment $\W$-semi-simple. 
Alors pour toute représentation lisse $\V$ de $\G$ qui est sous-quotient 
d'une re\-pré\-sentation engendrée par son composant iso\-ty\-pi\-que 
de type $\tau$, 
les $\Hh$-modules $\Hom_{\K}(\V,\W)$ et $\Hom_{\K}(\W,\V)^*$ 
sont isomorphes. 
\end{theo}

\begin{proof}
Une telle représentation $\V$ est sous-quotient 
d'une somme directe arbitraire de copies de $\ind^\G_\K(\tau)$. 
Le résultat est alors une conséquence des propositions 
\ref{WSS}, \ref{fWsssqs} et \ref{hermitien}.
\end{proof}

Il se peut que $\Hom_\K(\V,\W)$ et $\Hom_\K(\W,\V)^*$ 
soient des $\Hh$-modules à gauche 
isomorphes sans que cet isomorphisme ne provienne de 
l'accouplement \eqref{accVS}.  

\begin{exem}
\label{EX2}
Supposons que $\R$ est de caractéristique $\ell>0$, 
et supposons que $\G=\K$ est un groupe cyclique d'ordre $\ell$. 
Soit $(\pi,\V)$ une auto-extension non triviale du caractère trivial de $\G$
et soit $(\tau,\W)$ le caractè\-re trivial de $\K$.
L'accouplement \eqref{accVS} est nul~; 
pourtant, les $\Hh$-modules $\Hom_\K(\V,\W)$ et $\Hom_\K(\W,\V)^*$ 
sont tous deux isomorphes au caractère trivial de $\Hh$. 
\end{exem}

\subsection{}
\label{hyp2}

Revenons à la notion de $\W$-semi-simplicité définie au paragraphe \ref{S23}. 

\begin{prop}
\label{indWss}
Supposons que les hypothèses \ref{h3} et \ref{h4} sont vérifiées.
Alors l'induite com\-pac\-te 
$\ind^\G_\K(\tau)$ est $\W$-semi-simple.
\end{prop}

\begin{proof}
D'après la formule de Mackey et la remarque \ref{remaoplusWss}, 
pour que l'induite compacte $\ind^\G_\K(\tau)$ 
soit $\W$-semi-simple, il faut et il suffit que, pour chaque $g\in\G$, l'espace~:
\begin{equation*}
\I(g)=\ind^\K_{\K\cap\K^g}(\tau^g)
\end{equation*} 
le soit.
Étant donné $g\in\G$, l'hypothèse \ref{h4} implique (voir le lemme \ref{sym}) que~: 
\begin{equation*}
\Hom_{\K\cap\K^g}(\rho,\rho^g)\neq\{0\}
\quad\Leftrightarrow\quad
\Hom_{\K\cap\K^{g^{-1}}}(\tau,\tau^{g^{-1}})=\Hom_{\K\cap\K^g}(\rho^g,\rho)\neq\{0\}.
\end{equation*}
En d'autres termes, $\Hom_\K(\W,\I(g))$ est nul si et seulement si 
$\Hom_\K(\I(g),\W)$ est nul, auquel cas $\I(g)$ est $\W$-semi-simple.

Dans le cas où ces espaces sont non nuls, ce que nous supposons maintenant 
jusqu'à la fin de la preuve, 
ils sont de dimension $1$ d'après l'hypothèse \ref{h3}. 
Fixons un élément non nul $\T\in\Hh$ sup\-por\-té par la double classe 
$\K g\K$. 
Alors $\T(g)\in\Hom_{\K\cap\K^g}(\tau,\tau^g)$ est un opérateur 
d'entrelacement non nul et si l'on note $k\cdot w$ le vecteur 
$\tau(k)(w)$ pour $w\in\W$ et $k\in\K$, 
l'appli\-ca\-tion $\a\in\Hom_\K(\W,\I(g))$ défi\-nie par~: 
\begin{equation*}
\a(w)(k)=\T(g)(k\cdot w)
\end{equation*}
pour $w\in\W$ et $k\in\K$ 
est un générateur de $\Hom_\K(\W,\I(g))$.  
Parallèlement, d'après le lemme \ref{sym} à nou\-veau, 
$\T^{-1}(g^{-1})\in\Hom_{\K\cap\K^g}(\tau^g,\tau)$ 
est un opérateur d'entrelacement non nul. 
Si pour tout $w\in\W$ et tout $k\in\K$ 
on note $[k,w]$ l'élément de $\I(g)$ de support $(\K\cap\K^g)k$ et 
prenant en $k$ la valeur $w$, 
alors l'appli\-ca\-tion $\b\in\Hom_\K(\I(g),\W)$ définie par~: 
\begin{equation*}
\b([k,w])=k^{-1}\cdot\T^{-1}(g^{-1})(w)
\end{equation*}
pour $w\in\W$ et $k\in\K$ est un générateur de $\Hom_\K(\I(g),\W)$.
Ainsi $\I(g)$ est $\W$-semi-simple 
si et seulement si $\b\circ\a$ est un opérateur non nul de $\End_\K(\W)$. 
Pour tout $w\in\W$, on a~:
\begin{equation*}
\a(w)
=\sum\limits_{(\K\cap\K^g)\backslash\K}[x,\a(w)(x)]
=\sum\limits_{(\K\cap\K^g)\backslash\K}[x,\T(g)(x\cdot w)]
\end{equation*}
où $x$ décrit un système quelconque de représentants de 
$(\K\cap\K^g)\backslash\K$ dans $\K$, donc~:
\begin{equation*}
\b(\a(w))=\sum\limits_{(\K\cap\K^g)\backslash\K}
x^{-1}\cdot\T^{-1}(g^{-1})(\T(g)(x\cdot w))
\end{equation*}
 \ie que~:
\begin{equation*}
\b\circ\a=\sum\limits_{(\K\cap\K^g)\backslash\K}
\tau(x^{-1})\circ\T^{-1}(g^{-1})\circ\T(g)\circ\tau(x).
\end{equation*}
Ainsi, on trouve que $\b\circ\a$ 
est égal à $\T^{-1}*\T(1)={\rm id}_\W$.
\end{proof}

On peut se demander s'il y a des exemples où les hypothèses 
\ref{h3} et \ref{h4} sont vérifiées et où l'induite com\-pac\-te 
$\ind^\G_\K(\tau)$ n'est pas fortement $\W$-semi-simple.
Nous ne connaissons pas de tels exemples.
En revanche, nous avons le résultat suivant. 

\begin{prop}
\label{triv}
Supposons que $\R$ est de caractéristique $\ell>0$,
et supposons que $\tau$ est le caractère trivial de $\K$.  
Les conditions suivantes sont équivalentes. 
\begin{enumerate}
\item 
Pour tout $g\in\G$, l'indice de $\K\cap\K^g$ dans $\K$ est 
premier à $\ell$.
\item
L'induite com\-pac\-te $\ind^\G_\K(1)$ est $\W$-semi-simple. 
\end{enumerate}
\end{prop}

\begin{proof}
Remarquons d'abord que le caractère trivial de $\K$ vérifie automatiquement 
l'hypothèse \ref{h3}. 
Plus précisément, et avec les notations de la preuve de la proposition 
\ref{indWss}, l'unique sous-espace de $\I(g)$, pour $g\in\G$, qui soit 
invariant par $\K$ est celui des fonctions constantes, tandis que 
son unique (à un scalaire près) forme linéaire invariante par $\K$ 
est donnée par la somme des valeurs.
Il s'ensuit que $\I(g)$ est $\W$-semi-simple
si et seulement si  l'indice de $\K\cap\K^g$ dans $\K$ est premier à $\ell$.
\end{proof}


Voici deux exemples qui illustrent la proposition \ref{triv}.

\begin{exem}
Supposons que $\R$ est de caractéristique $\ell>0$. 
\begin{enumerate}
\item 
Si $\K$ est un pro-$\ell$-groupe, 
alors $\tau$ est le caractère trivial de $\K$ et 
$\ind^\G_\K(1)$ est $\W$-semi-simple si et seulement si $\K$ 
est normal dans $\G$.
\item 
Soit $\F$ un corps localement compact et non archimédien,
et soit $\G$ le groupe $\GL_n(\F)$ avec $n\>2$.
Supposons que $(\tau,\W)$ est le ca\-rac\-tère trivial d'un 
sous-groupe compact maximal $\K\subseteq\G$.
Alors l'induite $\ind^\G_\K(1)$ est $\W$-semi-simple 
si et seulement si le cardinal $q$ du corps résiduel de $\F$ 
et tous les $1+q+\dots+q^{k}$ avec $k\in\{1,\dots,n-1\}$
sont premiers à $\ell$. 
\end{enumerate}
\end{exem}

\subsection{}

Pour terminer cette section, on fait le lien entre la $\W$-semi-simplicité 
forte et la notion de quasi-projectivité introduite dans l'appendice de \cite{Vigs}. 

\begin{defi}
Une représentation lisse $\Q$ de $\G$ est \textit{quasi-projective} si, 
pour tout homomorphisme surjectif $f$ de $\Q$ dans une représentation $\V$, 
l'homomorphisme de $\End_\G(\Q)$-modules de $\End_\G(\Q)$ dans 
$\Hom_\G(\Q,\V)$ défini par $u\mapsto f\circ u$ est surjectif.
\end{defi}

\begin{prop}
\label{FSSiQPTF}
Supposons que l'induite compacte $\ind^\G_\K(\tau)$ est fortement 
$\W$-semi-simple.
Alors elle est quasi-projective et de type fini. 
\end{prop}

\begin{proof}
On note $\Q$ l'induite compacte $\ind^\G_\K(\tau)$. 
Soit $f:\Q\to\V$ un homomorphisme surjectif de représentations de $\G$.  
D'après la proposition \ref{fWsssqs}, 
$\V$ est fortement $\W$-semi-simple puisque $\Q$ l'est, et $f$ induit un 
homomorphisme surjectif $\K$-équivariant de $\Q(\tau)$ dans $\V(\tau)$. 
On en déduit un homomorphisme surjectif~:
\begin{equation*}
\Hom_\K(\W,\Q(\tau))\to\Hom_\K(\W,\V(\tau))
\end{equation*}
Le membre de droite est égal à $\Hom_{\K}(\W,\V)$, 
et le membre de gauche est égal à $\Hom_{\K}(\W,\Q)$, 
qui est canoniquement isomorphe à $\End_\G(\Q)$ par réciprocité de Frobenius. 
L'homomorphisme surjectif de 
$\End_\G(\Q)$ dans $\Hom_\G(\Q,\V)$ ainsi obtenu est bien $u\mapsto f\circ u$.
\end{proof}

\begin{rema}
La proposition \ref{FSSiQPTF} peut aussi être obtenue en utilisant 
\cite[Lemma 3.1]{VigEMS} et la notion de presque-projectivité de \cite{Vigs}. 
\end{rema}

Grâce aux théorèmes 4 et 10 de l'appendice de \cite{Vigs}, on déduit de la 
proposition \ref{FSSiQPTF} le résultat suivant. 

\begin{coro}
Si l'induite compacte $\ind^\G_\K(\tau)$ est fortement 
$\W$-semi-simple, alors le foncteur $\V\mapsto\Hom_\K(\W,\V)$ 
possède les propriétés suivantes~: 
\begin{enumerate}
\item
Il induit une bijection entre les classes d'isomorphisme de représentations 
irréductibles $\V$ de $\G$ telles que $\V(\tau)\neq\{0\}$ 
et les classes d'isomorphisme de $\Hh$-modules à droite simples. 
\item
Il est exact sur la sous-catégorie pleine des représentations lisses de $\G$ 
formée des repré\-sen\-tations qui sont sous-quotients de représentations 
engendrées par leur composant isotypique de type $\tau$.
\end{enumerate}
\end{coro}

\begin{rema}
La sous-catégorie pleine des représentations lisses de $\G$ qui sont 
sous-quo\-tients 
de représentations engendrées par leur composante isotypique de type $\tau$
n'est pas stable par ex\-ten\-sions en général~:
elle ne contient pas la représentation de l'exemple \ref{EX2}. 
\end{rema}

\section{Application aux types simples}
\label{S4}

On suppose dans ce paragraphe que $\G$ est le groupe $\GL_m(\D)$, 
où $m\>1$ et où $\D$ est une algèbre à division centrale et de dimension 
finie sur un corps localement compact et non archimédien $\F$ de 
carac\-téristique résiduelle $p$. 
On suppose que $\R$ est algébriquement clos et de 
carac\-téristique différente de $p$, et que $(\K,\tau,\W)$ est un type 
simple de $\G$ au sens de \cite{MSt}, §2.5.

D'après \cite{MSt}, lem\-me 2.19 et co\-rol\-laire 2.23, 
les hypothèses \ref{h3} et \ref{h4} du paragra\-phe \ref{hyp12enonce} 
sont satis\-faites pour $(\K,\tau,\W)$, 
et il découle de la preuve de \cite{MSt}, proposition 4.9 
que l'induite compacte $\ind^\G_\K(\tau)$ est for\-tement 
$\W$-semi-simple.
On déduit du théorème \ref{Herm} le résultat suivant. 

\begin{theo}
\label{thtss}
Soit $\V$ un sous-quotient d'une 
représentation de $\GL_m(\D)$ engen\-drée par sa compo\-sante isotypique
de type $\tau$.
Alors les $\Hh$-modules à 
gauche $\Hom_\K(\V,\W)$ et $\Hom_\K(\W,\V)^*$ sont isomorphes. 
\end{theo}

Ce théorème permet de montrer que la classification de Zelevinski 
$\m\mapsto\Z(\m)$ 
des repré\-sen\-ta\-tions 
lisses irréductibles de $\GL_m(\D)$ à coefficients dans $\R$ en termes
de multisegments \cite{MSc} 
est com\-pa\-tible au passage à la contra\-gré\-diente, \ie que si 
$\m$ est un multisegment et $\m^\vee$ son multisegment contragrédient, 
alors $\Z(\m^\vee)$ est isomorphe à la contragrédiente de $\Z(\m)$.

\providecommand{\bysame}{\leavevmode ---\ }
\providecommand{\og}{``}
\providecommand{\fg}{''}
\providecommand{\smfandname}{\&}


\providecommand{\bysame}{\leavevmode ---\ }
\providecommand{\og}{``}
\providecommand{\fg}{''}
\providecommand{\smfandname}{\&}
\providecommand{\smfedsname}{\'eds.}
\providecommand{\smfedname}{\'ed.}
\providecommand{\smfmastersthesisname}{M\'emoire}
\providecommand{\smfphdthesisname}{Th\`ese}
\begin{thebibliography}{}

\end{thebibliography}


\begin{thebibliography}{10}

\bibitem{BourbakiALG}
{\scshape N.~Bourbaki}, 
Algèbre, Chapitre 8.
Seconde édition, Springer-Verlag, Berlin, 2012. 

\bibitem{BourbakiENS}
\bysame , 
Théorie des ensembles.
Springer-Verlag, Berlin, 2006. 

\bibitem{BKl}
{\scshape C.~J. Bushnell {\normalfont \smfandname} P.~C. Kutzko}, 
{\og The admissible dual of {${\rm GL}({\rm N})$} via compact open 
  sub\-groups\fg}, 
  Princeton University Press, Princeton, NJ, 1993.

\bibitem{BKt}
\bysame , {\og Smooth representations of reductive $p$-adic groups: structure
  theory via types\fg}, \emph{Proc. London Math. Soc. (3)} \textbf{77} (1998),
  no.~3, p.~582--634.

\bibitem{Herzig}
{\scshape F.~Herzig}, {\og The classification of irreducible admissible mod 
  $p$ representations of a $p$-adic ${\rm GL}_{n}$\fg}, 
\emph{Invent.  Math.} {\bf 186} (2011), n°2, p.~373--434.

\bibitem{MSt}
{\scshape A. M\'{\i}nguez \& \V.~S{\'e}cherre}, 
{``Types modulo $\ell$ pour les formes intérieures de ${\rm GL}_n$ 
sur un corps lo\-cal non archimédien''}, 
pré\-pu\-bli\-cation, voir
\verb!http://lmv.math.cnrs.fr/annuaire/vincent-secherre/!.

\bibitem{MSc}
{\scshape A. M\'{\i}nguez \& \V.~S{\'e}cherre}, 
{``Représentations lisses modulo $\ell$ de ${\rm GL}_m(\D)$''}, 
pré\-pu\-bli\-cation, voir
\verb!http://lmv.math.cnrs.fr/annuaire/vincent-secherre/!.

\bibitem{Tad}
{\scshape M.~Tadi{\'c}}, {\og Induced representations of {${\rm GL}(n,A)$}
  for {$p$}-adic division algebras {$A$}\fg}, \emph{J. Reine Angew. Math.}
  \textbf{405} (1990), p.~48--77.

\bibitem{Vigb}
{\scshape M.-F. Vign{\'e}ras}, {\og Repr\'esentations {$l$}-modulaires d'un
  groupe r\'eductif {$p$}-adique avec {$l\ne p$}\fg}, Progress in Mathematics,
  vol. 137, Birkh\"auser Boston Inc., Boston, MA, 1996.

\bibitem{Vigs}
\bysame , {\og Induced {$R$}-representations of {$p$}-adic reductive
  groups\fg}, \emph{Selecta Math. (N.S.)} \textbf{4} (1998), no.~4,
  p.~549--623.
With an appendix by Alberto Arabia.

\bibitem{VigEMS}
\bysame , {\og Irreducible modular representations of a reductive {$p$}-adic
  group and simple modules for {H}ecke algebras\fg}, in \emph{European Congress
  of Mathematics, Vol. I (Barcelona, 2000)}, Progr. Math., vol. 201,
  Birkh\"auser, Basel, 2001, p.~117--133.

\bibitem{Zel}
{\scshape A.~V. Zelevinsky}, {\og Induced representations of reductive
  {${\mathfrak{p}}$}-adic groups. {II}. {O}n irreducible representations of
  {${\rm GL}(n)$}\fg}, \emph{Ann. Sci. \'Ecole Norm. Sup. (4)} \textbf{13}
  (1980), no.~2, p.~165--210.

\end{thebibliography}
\end{document}